\documentclass[12pt]{article}
\usepackage{amssymb, amsmath, amsthm, amscd}
\usepackage[dvips]{graphics}
\usepackage[utf8]{inputenc}
\usepackage[all,cmtip]{xy}
\usepackage{dsfont}
\usepackage{dsfont}
\usepackage{mathrsfs}  
\usepackage{bbm}
\usepackage{textcomp}
\usepackage{eurosym}
\usepackage{enumitem}
\usepackage{setspace}
\usepackage{color}
 \usepackage{setspace}
\usepackage{color}
\usepackage[usenames,dvipsnames,svgnames,table]{xcolor}
\usepackage[colorlinks=true,
            linkcolor=blue,
            urlcolor=blue,
            citecolor=blue]{hyperref}
\usepackage{filecontents}

\newtheorem{lem}{Lemma}[section]
\newtheorem{pro}[lem]{Proposition}
\newtheorem{defi}[lem]{Definition}
\newtheorem{def/not}[lem]{Definition/Notations}

\newtheorem{thm}[lem]{Theorem}
\newtheorem{mthm}[lem]{Main Theorem}

\newtheorem{cor}[lem]{Corollary}

\newtheorem{rqe}[lem]{Remarks}
\newtheorem{exa}[lem]{Example}

\newcommand{\C}{\mathbb{C}}

\newcommand{\N}{\mathbb{N}}

\begin{document}

\begin{center}

{\Large\bf  Subextension and approximation of $m$-subharmonic functions with given boundary values}

\end{center}

\begin{center}

\large Hichame Amal \footnote{{Department of mathematics, Laboratory LaREAMI, Regional Centre of trades of education and
training, Kenitra Morocco E-mail: hichameamal@hotmail.com}} 
 and  \large Ayoub El Gasmi \footnote{{Ibn tofail university, faculty of sciences, department of mathematics, PO 242 Kenitra Morroco
E-mail: ayoub.el-gasmi@uit.ac.ma}}

\end{center}

\author{}

\vspace{2ex}

\begin{abstract} In this paper we establish a result on subextension of $m$-subharmonic functions  in the class $\mathcal{F}_m(\Omega,f)$  without changing the hessian measures. As application, we approximate an $m$-subharmonic function with given boudary value by an
increasing sequence of $m$-subharmonic functions defined on larger domains.
\\ 
    
    \vspace{1ex}
    
\end{abstract}

\section{Introduction}
Let $\Omega$ and $\tilde{\Omega}$ be an $m$-hyperconvex domains of $\C^{n}$, $1\leq m\leq n$. The purpose of this paper is to study the subextension of $m$-subharmonic functions belonging to the class  $\mathcal{F}_m(\Omega,f)$ (see the next section for definition). We say that an $m$-subharmonic function $u$ on $\Omega$ ($u\in \mathcal{SH}_m(\Omega)$ for short) has a  \textit{subextension} if there exists a function $\tilde{u}\in\mathcal{SH}_m(\tilde{\Omega})$ such that  $\Omega\subset \tilde{\Omega}$ and $\tilde{u}\leq u$ in $\Omega$.
\par The case $m=n$ has been studied by many authors over the last forty years. EL Mir \cite{E} gave in 1980 an example of a plurisubharmonic function on the unit bidisc in $\mathbb{C}^2$ for which the restriction to any smaller bidisc admits no subextension to a larger domain. 
 He also proved that, after attenuatting the singularities of a given plurisubharmonic function by composition with a suitable convex increasing
function, it is possible to obtain a global subextension. Alexander and Taylor \cite{AT} gave in $1984$ a generalization of this result.
In $2003$, Cegrell and Zeriahi proved in \cite{CZ} that plurisubharmonic functions in the Cegrell class $\mathcal{F}(\Omega)$ admit a plurisubharmonic subextension to any larger bounded hyperconvex domain with a control of the
Monge-Ampère mass. The subextension fails  in general in the class $\mathcal{E}(\Omega)$, even in $\mathcal{N}(\Omega)$, as shown by Wiklund and  Lisa Hed in \cite{W} and \cite{H}. 
For the energy and weighted energy classes $\mathcal{E}_p(\Omega)$, $\mathcal{E}_\chi(\Omega)$, and  $\mathcal{E}^{\psi}(\Omega)$ the subextension problem was investigated respectively by Hiep, Benelkourchi, Hai and Tang Van Long (see \cite{Hi}, \cite{Ben1}, \cite{HL})
\par
    The problem of subextension of psh functions with boundary value, especially in the the class $\mathcal{F}(\Omega, H)$, has been studied by C\.zyz and Hed in \cite{CH}, by Hed in \cite{H}, and by Åhag and C\.zyz in \cite{AC}. Later on, Le Mau Hai and Nguyen Xuan Hong poved in \cite{MN} that one can give a subextension of a function in the class without changing the Monge-Ampere measure. In \cite{A}, the first author studied the extension in the class $\mathcal{E}_{\chi}(\Omega, H)$.
\par A few years ago, subextension of $m$-subharmonic functions in the class $\mathcal{F}_{m,\chi}(\Omega)$ was studied  by Hung in \cite{VV}, who proved that an $m$-subharmonic function in $\mathcal{F}_{m,\chi}(\Omega)$, under the assumption that the $m$-hyperconvex domains $\Omega$ and $\tilde{\Omega}$ are  such that $\Omega$ is relatively compact in $ \tilde{\Omega},$ have a subextension in $\mathcal{F}_{m,\chi}(\tilde{\Omega})$. This result has been ameliorated recently by Mau hai and Dung in \cite{MV}, they proved the existence of a subextension in $\mathcal{F}_m(\Omega)$ where $\Omega$ does not have to be relatively compact in $\tilde{\Omega}$, they further showed an equality of the complex Hessian
measures of subextension and initial function. 
\par Motivated an inspired with the above results and techniques, our aim in this paper is to study the  subextension problem in the class $\mathcal{F}_m(\Omega,f).$ Namely, we prove the following:
\begin{mthm}
  Let $\Omega$ and $\tilde{\Omega}$ be a $m$-hyperconvex domains such that $\Omega\subset\tilde{\Omega}$. Given $f\in\mathcal{E}_m(\Omega)$ and $g\in\mathcal{MSH}_m(\tilde{\Omega})\cap\mathcal{E}_m (\tilde{\Omega})$ that satisfy $g \leq f$ on $\Omega$. If $u\in \mathcal{F}_m(\Omega,f)$ and $\int_{\Omega}H_m(u) < +\infty$,
then there exists a function $\tilde{u}\in\mathcal{F}_m(\tilde{\Omega},g)$ such that $\tilde{u}\leq u$  on $\Omega$, and $H_m(\tilde{u}) =
\mathds{1}_{\Omega}H_m(u).$
\end{mthm}
As application of this result and using an approximation Theorem in the classes $\mathcal{F}_m(\Omega)$ proved  recently by Nguyen Van Phu and Nguyen Quang Dieu in \cite{PD}, 
we give the following Theorem:
\begin{thm}
Let $\Omega\subset\Omega_{j+1}\subset\Omega_j$
be an $m$-hyperconvex domains such that $\lim_{j\to +\infty}\mathrm{cap}_m(\Omega_j,K)=\mathrm{cap}_m(\Omega,K)$ for all compact subset
$K\subset\Omega$, and that $g\in \mathcal{MSH}_m^{-}(\Omega_1)$. Then, to every function
$u\in \mathcal{F}_m(\Omega,g_{|\Omega})$, such that
$\int_{\Omega}H_m(u) < +\infty$,
there exists an increasing sequence of functions $u_j\in\mathcal{F}_m(\Omega_j,g_{|\Omega_j})$ such that
$\lim_{j\to +\infty} u_j=u$ almost everywhere on $\Omega$.
\end{thm}


\par The paper is organized as follows. In a preliminary section (section \ref{sec2}), we give some results about $m$-subharmonic functions, Cegrell classes and maximality. The section \ref{sec3} is devoted to the proof of the subextension Theorem. We give an example of a $m$-sh function in the class $\mathcal{N}_m(\Omega)\setminus\mathcal{F}_m(\Omega)$  which has no subextension.  Finally, in section \ref{sec4}, we will prove the approximation Theorem.
\section{Prelimenaries}\label{sec2}
\subsection{The class of $m$-subharmonic functions}
In this subsection, we recall briefly some basic properties of admissible functions for the complex Hessian operator, such functions are called $m$-subharmonic ($m$-sh for short) and $\mathcal{SH}_m(\Omega)$ denote the set of all such functions. They are in particular subharmonic and non-smooth in general. 
\par Let $1\leq m\leq n.$ The elementary symmetric function $S_m$ is defined by 
$$S_m(\lambda)=\sum_{1\leq j_1<\cdots< j_m\leq n}\lambda_{j_{1}}\ldots\lambda_{j_{m}},$$
 which can be also determined by
                 $(\lambda_1+t)\cdots(\lambda_n+t)=\sum_{m=0}^{n}S_m(\lambda)t^{n-m}, \;\hbox{with}\;\;\; t\in \mathbb{R}.$
Let $\Gamma_m$ be the closure of the  connected component of $\{\lambda\in \mathbb{R}^n:\; S_m(\lambda)>0\}$ containing $(1,\cdots,1).$ 

Let $(a_{j\bar{k}})$ be any  complex hermitian matrix $n\times n$ and  let $\alpha=\dfrac{i}{2}\sum_{j,k}a_{j\bar{k}}dz_j\wedge d\bar{z}_k$
be the correspondent real differential form of bidegree $(1,1)$. Denote by  $\lambda(\alpha)=(\lambda_1(\alpha),\cdots,\lambda_n(\alpha))\in \mathbb{R}^{n}$, the vector of the eigenvalues of $(a_{j\bar{k}})$. One can define  
                                $\tilde{S}_m(\alpha)=S_m(\lambda(\alpha)).$
This relation between $(1,1)$-forms and vectors in $\mathbb{R}^{n}$ allows us to define the set
$$\tilde{\Gamma}_m:=\{\alpha\in\mathbb{C}_{(1,1)}:\; \lambda(\alpha)\in \Gamma_m\}=\{\alpha:\;\tilde{S}_k(\alpha)\geq0, \forall \;1\leq k\leq m\},$$
where $\mathbb{C}_{(1,1)}$ is the space of real $(1,1)$-forms with constant coefficients in $\mathbb{C}^n$. After diagonalizing the matrix $(a_{j\bar{k}})$ we see that
           $\alpha^m\wedge \beta^{n-m}=\frac{m!(n-m)!}{n!}\tilde{S}_m(\alpha)\beta^{n}.$
We define
   $\breve{\Gamma}_m:=\{\alpha\in \mathbb{C}_{1,1}:\; \alpha\wedge \beta^{n-1}\geq0,\alpha^2\wedge \beta^{n-2}\geq0,\cdots,\alpha^m\wedge \beta^{n-m}\geq0\}.$
 A $(1,1)$-form belonging to $\breve{\Gamma}_m$ is called $m$-positive.
 If $T$ is a current  of bidegree $(n-k, n-k)$, $k \leq m$, then $T$ is called $m$-positive if for all $m$-positive $(1,1)$-forms $\alpha_1,...,\alpha_{k}$ we have 
$\alpha_1\wedge\alpha_2\wedge\ldots\wedge\alpha_k\wedge T\geq 0.$\\
In connection with the results above, we recall the definition and some basic propreties of $m$-sh functions. Here $d=\partial+\overline{\partial}$ is the standard operator of exterior differentiation while $d^{c}=\dfrac{1}{2i\pi}(\partial-\overline{\partial}).$
\begin{defi}
Let $u$  be a subharmonic function defined on a bounded domain $\Omega\subset\mathbb{C}^{n}.$
\begin{itemize}
  \item[$i)$] If $u\in C^2(\Omega)$, then $u$ is called $m$-sh if $dd^cu$ belongs pointwise to $\breve{\Gamma}_m.$
  \item[$ii)$] For non-smooth case, $u$ is  $m$-sh if for all $\alpha_1,\ldots,\alpha_{m-1}\in\breve{\Gamma}_m,$ the  inequality
   $$dd^cu\wedge\alpha_1\ldots \wedge \alpha_{m-1}\wedge \beta^{n-m}\geq0,\;\;\;\;\;\;$$
holds in the weak sense of currents in $\Omega$.
\end{itemize}
\end{defi}
\begin{defi}
 We say that a function $u$ is strictly $m$-sh on $\Omega$ if it is $m$-sh on $\Omega$, and for every $P\in\Omega$ there exists a constant  $C_P$ such that the function $z\longmapsto u(z)-C_{P}\vert z\vert ^{2}$ is $m$-sh in a neighborhood of $P$.
\end{defi}
\begin{pro}\cite{Blo}.
\begin{enumerate}
  \item[$1)$] $\mathcal{PSH}=\mathcal{SH}_{n}\subset\mathcal{SH}_{n-1}\subset\ldots\subset\mathcal{SH}_1\subset \mathcal{SH}$.
  \item[$2)$] If $u, v\in \mathcal{SH}_m(\Omega)$ then $\lambda u+ \mu v \in \mathcal{SH}_m(\Omega)$, $\forall \lambda, \mu\geq 0.$
  \item[$3)$] The limit of a decreasing sequence of $m$-sh functions is an $m$-sh function.
  \item[$4)$] If $u\in \mathcal{SH}_m(\Omega)$ and $f$ is a convex increasing function, then $f\circ u\in \mathcal{SH}_m(\Omega).$
  \item[$5)$] If $u\in \mathcal{SH}_m(\Omega)$, then the standard regularization $u\ast\rho_{\epsilon}\in\mathcal{SH}_m(\Omega_{\epsilon})$, where $\Omega_{\epsilon}:=\{z\in\Omega:\; dis(z,\partial\Omega)>\epsilon) \}$, for $0<\epsilon\ll1.$
  \item[$6)$] If $(u_j)\subset\mathcal{SH}_m(\Omega)\cap L_{\ell oc}^{\infty}(\Omega),$ then $(\sup_j u_j)^{*}\in\mathcal{SH}_m(\Omega)$, where $\theta^{*}$ denotes the upper semicontinuous regularisation of $\theta.$
\end{enumerate}
\end{pro}

For locally bounded $m$-sh functions $u_1,\ldots,u_k$, $k\leq m,$ defined on a bounded domain $\Omega\subset\mathbb{C}^{n}$, one can inductively define a closed nonnegative current as
$dd^cu_1\wedge\ldots\wedge dd^cu_k\wedge\beta^{n-m}:=dd^c(u_1dd^cu_2\wedge\ldots\wedge dd^cu_k\wedge\beta^{n-m}).$
In particular, for a function $u\in \mathcal{SH}_m(\Omega)\cap L_{\ell oc}^{\infty}(\Omega)$, the complex $m$-Hessian measure  is defined by:
$$H_m(u)=(dd^cu)^m\wedge\beta^{n-m}.$$
\begin{defi}
Let $\Omega$ be a bounded domain in $\mathbb{C}^{n}$. 
\begin{itemize}
\item[$1)$]  We say that $\Omega$ is $m$-hyperconvex if there exists a continuous $m$-sh function $\varphi:\; \Omega \rightarrow \mathbb{R}^{-}$ such that $\{\varphi<c\}\Subset\Omega$, for every $c<0.$
\item[$2)$]  We say that $\Omega$ is strongly $m$-hyperconvex if there exists an $m$-sh function $\rho$ defined on some open neighborhood $\Omega^{\prime}$ of $\bar \Omega$ such that $\Omega= \{z\in\Omega^{\prime}:\; \rho(z) < 0\}$.
\item[$3)$]  We say that $\Omega$ is strictly $m$-pseudoconvex if there exists a smooth strictly $m$-sh function $\rho$ on some open neighborhood $\Omega^{\prime}$ of $\bar \Omega$ such that $\Omega= \{z\in\Omega^{\prime}:\; \rho(z) < 0\}$.
\end{itemize}
\end{defi}
\subsection{Cegrell's classes for $m$-subharmonic functions}
\par From now on, $\Omega$ is assumed to be an $m$-hyperconvex domain of $\mathbb{C}^{n}.$
\begin{itemize}
    \item[$\bullet$] We denote $\mathcal{E}^0_m(\Omega)$  the class of bounded $m$-sh functions $u$  on $\Omega$ such that
$\displaystyle\lim_{z\rightarrow \xi}u(z)=0$, $\forall\xi\in\partial\Omega$ and $\int_{\Omega}H_m(u)<+\infty.$
\item[$\bullet$] Let $u \in \mathcal{SH}_{m}^{-}(\Omega)$, we say that $u$  belongs to $\mathcal{E}_m(\Omega)$ if for each $z_0 \in \Omega$, there exists an open neighborhood $U \subset \Omega$ of $z_0$ and a decreasing sequence $(u_j)$ in $\mathcal{E}_m^0$ such that 
  $u_j \downarrow u$ on $U$ and $\sup_j\int_\Omega H_m(u_j)<+\infty.$
  \item[$\bullet$]  We denote by $\mathcal{F}_m(\Omega)$ the class of functions $u \in \mathcal{SH}_{m}^{-}(\Omega)$ such that there exists a sequence $(u_j) \subset \mathcal{E}_m^0(\Omega)$  decreasing to $u$ in $\Omega$ and $\sup_j\int_\Omega H_m(u_j)<+\infty.$
\end{itemize}
Let $u\in\mathcal{E}_m(\Omega)$, it follows from \cite[Theorem 3.14]{Ch1} that if $(u_j)$ is sequences in $\mathcal{E}_{m}^{0}(\Omega)$, decreasing to $u$, then the sequence of measures $H_m(u_{j})$ converges weakly to a positive Radon measure which does not depend on the choice of the sequence $(u_{j})_j.$ One then can define $H_m(u)$ to be this weak limit. Therefore, if $u\in\mathcal{E}_m(\Omega)$, then  $H_m(u)$ is well defined and it is a Radon measure on $\Omega$. Note also that the class $ \mathcal{E}_m(\Omega)$ is the largest  subset of $\mathcal{SH}_{m}^{-}(\Omega)$ on which the complex $m$-Hessian operator $H_m$ is
well defined and continuous under decreasing limits (see \cite{Ch1}).
\begin{defi}{\textbf{(Maximal $m$-sh functions).}}
A function $u\in\mathcal{SH}_m(\Omega)$ is called maximal if for any domain $G\Subset \Omega,$ the inequality  $v\leq u$ holds in $G$ for all $v\in\mathcal{SH}_m(\Omega)$ satisfying  $v_{\vert\partial G}\leq u_{\vert\partial G}.$ \end{defi}
\par We denote by $\mathcal{MSH}_{m}(\Omega)$ the family of all these functions. Note that it follows from  \cite{Blo} and \cite{V} that an $m$-sh function $u$ belonging to $\mathcal{E}_m(\Omega)$ is maximal if and only if $H_{m}(u)=0.$ For $m=1$, maximal subharmonic functions coincide with harmonic functions.
\begin{defi}{\textbf{($m$-polar sets).}}
A set $P\subset\mathbb{C}^{n}$ is called $m$-polar if for any $z\in P$ there exists
a neighborhood $V$ of $z$ and a function $u\in \mathcal{SH}_m(V)$ such that $P\cap V\subset \{u=-\infty\}.$
\end{defi}
\par We give the definition and some propreties of the class  $\mathcal{N}_m(\Omega)$ that was first defined in \cite{G2} and originate from \cite{Ceg2}.
From now on, $(\Omega_j)$ is assumed to be a fundamental sequence of $\Omega$, this is a sequence of
strictly $m$-pseudoconvex domains such that $\Omega_j\Subset\Omega_{j+1}$ for every $j\in \N$ and $\cup_j \Omega_j=\Omega.$
\begin{defi}\label{120}
Let $u \in \mathcal{SH}_{m}(\Omega)$. Define $u^j:=\sup\left\{\phi \in \mathcal{SH}_{m}(\Omega):\; \phi\leq u\; \hbox{on}\;\Omega\setminus\Omega_j\right\},$
and let $\widetilde{u}:=(\displaystyle\lim_{j\rightarrow+\infty}u^j)^{*}.$
\end{defi}
Note that the definition of $u^{j}$ is independent of the exhaution $(\Omega_j)$ and that $u\leq u^{j}\leq u^{j+1}$, hence  $\sup_{j} u^{j}=\displaystyle\lim_{j\rightarrow+\infty}u^j$, which yields that $\widetilde{u}=(\displaystyle\lim_{j\rightarrow+\infty}u^j)^{*}\in\mathcal{SH}_{m}(\Omega).$ Moreover, if $u\in \mathcal{E}_{m} $ then $\widetilde{u}\in\mathcal{E}_{m}$.
\begin{pro}\cite[Poroposition 2.7]{G2}
Let $u\in\mathcal{E}_m(\Omega)$, then $\tilde{u}$ is maximal on $\Omega$. It is the smallest maximal $m$-sh majorant of $u$ in $\Omega$.
\end{pro}

\begin{rqe}\label{b}
\begin{enumerate}
\item[$i)$]  Let $u, v \in\mathcal{E}_{m}(\Omega)$ and $\alpha\in \mathbb{R}^{+}$, then $\widetilde{u+v}\geq \widetilde{u}+\widetilde{v}$ and  $\widetilde{\alpha u}=\alpha\widetilde{u}$. Moreover if  $u\leq v$ then $\widetilde{u}\leq\widetilde{v}.$
\item[$ii)$] $\mathcal{E}_{m}(\Omega)\cap \mathcal{MSH}_{m}(\Omega)=\{u\in\mathcal{E}_{m}(\Omega):\;\widetilde{u}=u\}.$
\end{enumerate}
\end{rqe} 
\begin{defi}
Let $u\in\mathcal{E}_{m}(\Omega),$ we say that $u\in\mathcal{N}_m(\Omega)$ if $\widetilde{u}=0.$
\end{defi}
By the remark above we derive  that $\mathcal{N}_m(\Omega)$ is a convex cone. Note also that we have the following inclusions 
$\mathcal{F}_m(\Omega)\subset \mathcal{N}_m(\Omega)\subset \mathcal{E}_m(\Omega).$
\begin{defi}
Let $\mathcal{K}_m(\Omega) \in \{\mathcal{E}_{m}^{0}(\Omega), \mathcal{F}_m(\Omega), \mathcal{N}_m(\Omega)\}$, and $f\in \mathcal{E}_m(\Omega)$. We say that an $m$-sh function $u$ defined on $\Omega$ belongs to $\mathcal{K}_m(\Omega,f)$ if there exists a function $\varphi \in \mathcal{K}_m(\Omega)$ such that $f\geq u\geq \varphi+f.$  
\end{defi}
Note that $\mathcal{K}_m(\Omega,0)=\mathcal{K}_m(\Omega)$ and that the complex Hessian measure of a function belonging to the class $\mathcal{K}_m(\Omega,f)$ does not necessarily have to be finite. Observe also that we have the following inclusions 
$$\mathcal{E}^{0}_m(\Omega,f)\subset\mathcal{F}_m(\Omega,f)\subset\mathcal{N}_m(\Omega,f)\subset\mathcal{E}_m(\Omega).$$
We need in the sequel the following  identity principle in the class $\mathcal{N}_{m}(f)$.
\begin{thm}\label{31}(\textbf{The Identity Principle}).
Let $f \in \mathcal{E}_m.$ If $u,v\in\mathcal{N}_m(f)$ such that $u\leq v$, $H_m(u)=H_m(v)$ and $\displaystyle\int_{\Omega}(-\omega)H_m(u)<+\infty$ with $\omega \in\mathcal{E}_m$, then $u=v$ on $\Omega$.
\end{thm}
The proof is an immediate consequence of the following lemma
\begin{lem}\label{80}
Let $f \in \mathcal{E}_m$ and $u, v \in \mathcal{N}_m(f)$ such that $u\leq v.$ Then for all
$\omega_j \in \mathcal{SH}_m(\Omega)\cap L^{\infty}(\Omega),$ $-1\leq \omega_j\leq0,$ $j=1,...,m$, $\displaystyle\int_{\Omega}(-\omega_1)H_m(u)<+\infty$, we have that the following inequality holds:
\begin{equation}\label{20}
  \frac{1}{m!}\displaystyle\int_{\Omega}(v-u)^{m}dd^{c}\omega_{1}\wedge ...\wedge dd^{c}\omega_{m}\wedge \beta^{n-m}
 + \displaystyle\int_{\Omega}(-\omega_{1})H_m(v)\leq  \displaystyle\int_{\Omega}(-\omega_{1})H_m(u).
\end{equation}
\end{lem}
\begin{proof}
Suppose first that $u, v \in \mathcal{E}^0_m(f).$ Then by definition there exists a  function $\varphi \in \mathcal{E}^0_m$ such that
$f\geq u\geq  f+\varphi.$ For each $\varepsilon>0$ small  enough we can choose $K\Subset\Omega$ such that $\varphi\geq -\varepsilon$ on $\Omega\setminus K$. Hence,
                $$u\geq\varphi+f\geq-\varepsilon+f\geq-\varepsilon+v \;\;\;\hbox{on}\;\;\; \Omega\setminus K,$$
Put $\widehat{u}=\max(u,v-\varepsilon)$, then $\widehat{u}=u$ on $\Omega\setminus K.$ It follows from  \cite[(2) in Lemma 5.4]{V} that
\begin{equation*}
\frac{1}{m!}\displaystyle\int_{\Omega}(\widehat{u}-u)^{m}dd^{c}\omega_{1}\wedge ...\wedge dd^{c}\omega_{m}\wedge \beta^{n-m}
+ \displaystyle\int_{\Omega}(-\omega_{1})H_m(\widehat{u})\leq \displaystyle\int_{\Omega}(-\omega_{1})H_m(u).
\end{equation*}
By letting $\varepsilon\rightarrow 0^+$ (\ref{20}) holds.
\par  Let now $u, v \in \mathcal{N}_m(f),$ then by Proposition \cite[Proposition 2.12]{G2} there exist two decreasing sequences $(u_j)$, $(v_k)$ such that $u_j, v_k \in\mathcal{E}^0_m(f)$ that converge pointwise to $u$ and $v$ respectively. By taking $\max\{u_j,v_j\}$ we can suppose $u_j\leq v_j$. Take $k\leq j$, then $v_k\geq v_j\geq u_j$,  so using the first part we get
 \begin{equation*}
\frac{1}{m!}\int_{\Omega}(v_k-u_j)^{m}dd^{c}\omega_{1}\wedge ...\wedge dd^{c}\omega_{m}\wedge \beta^{n-m}
 + \int_{\Omega}(-\omega_{1})H_m(v_k)\leq \int_{\Omega}(-\omega_{1})H_m(u_j).
\end{equation*}
Letting $j$ tends to $\infty$ and applying \cite[Proposition 2.13]{G2} we get 
\begin{equation*}
\frac{1}{m!}\int_{\Omega}(v_k-u)^{m}dd^{c}\omega_{1}\wedge ...\wedge dd^{c}\omega_{m}\wedge \beta^{n-m}
 + \int_{\Omega}(-\omega_{1})H_m(v_k)\leq \int_{\Omega}(-\omega_{1})H_m(u).
\end{equation*} 
Finally by letting $k$ tends to $\infty$ and  applying again \cite[Proposition 2.13]{G2} we get the desired inequality
\end{proof}

\section{Subextension in the class $\mathcal{F}_m(\Omega,f)$}\label{sec3}
\subsection{Subextension}
In this section we give the proof of our main Theorem. We
need the following two Lemmas.
\begin{lem}\label{lem1}
    Let $f\in\mathcal{E}_m(\Omega)$. If $u\in\mathcal{F}_m(\Omega,f)$ is such that 
$\int_{\Omega}H_m(u)<+\infty$, 
then there exists a decreasing sequence $(u_j)_j\subset\mathcal{E}_m^0(\Omega,f)$ that converges pointwise
to $u$ as $j$ tends to $+\infty$, and $\sup_j\int_{\Omega}H_m(u_j)<+\infty.$
Conversely, if $(u_j)_j\subset\mathcal{F}_m(\Omega,f)$ is a decreasing sequence that converges pointwise
to a function $u$, and if $\sup_j\int_{\Omega}H_m(u_j)<+\infty$, then $u \in\mathcal{F}_m(\Omega,f)$
with $\int_{\Omega}H_m(u)<+\infty$.
\end{lem}
\begin{proof}
   Suppose that  $u\in\mathcal{F}_m(\Omega,f)$ is such that with $\int_{\Omega}H_m(u)<+\infty$. Then by \cite[Proposition 2.12]{G2}, there exists a decreasing sequence $(u_j)_j\subset\mathcal{E}_m^0(\Omega,f)$ that
 converges pointwise to $u$ on $\Omega$. Since $\int_\Omega(dd^cu)^m\wedge\beta^{n-m}<+\infty$, if follows from \cite[Proposition 2.13]{G2} that $\sup_j\int_\Omega H_m(u_j)<+\infty.$
\par Conversely, assume that $(u_j)_j\subset\mathcal{F}_m(\Omega,f)$ be a decreasing sequence that converges pointwise
to a function $u$ as $j$ tends to $+\infty$ with $\sup_j\int_{\Omega}H_m(u_j)<+\infty$. 
 Assume first that $(u_j)_j\subset\mathcal{E}_m^0(\Omega,f)$. Since $u_j\leq f$, then by \cite[Lemma 2.14]{G2} we have $\int_\Omega H_m(f)<+\infty$ and by \cite[Theorem 4.8]{V} we get $f\in \mathcal{F}_m(\Omega,\tilde{f})$. Hence, we can without loss of generality assume that $H_m(f)=0.$ Now, since $H_m(u_j)$ vanishes on pluripolar sets and $\int_\Omega H_m(u_j)<+\infty$, then by \cite[Theorem 1.2]{Ch1} there exists $\psi_j\in \mathcal{F}_m(\Omega)$ such that $H_m(\psi_j)=H_m(u_j)$. But $H_m(\psi_j+f)\geq H_m(u_j)$, then by \cite[Corollary 3.3]{G2} it follows that $u_j\geq\psi_j+f.$ Put $\psi_j^{'}=(\sup_{k\geq j}\psi_k)^{*}$. Then $(\psi'_j)_j$
is a decreasing sequence  which belongs to $\mathcal{F}_m(\Omega)$. It follows from the comparison principle that 
$$\sup_j\int_{\Omega}H_m(\psi^{'}_j)\leq\sup_j\int_{\Omega}H_m(\psi_j)=\sup_j\int_{\Omega}H_m(u_j)<+\infty.$$  Let $\psi=\lim_{j\to +\infty}\psi_j^\prime$. By \cite[Lemma 4.7]{V} we have $\psi\in\mathcal{F}_m(\Omega)$. Let $j\in\mathbb{N}$. Since $u_j\geq u_k\geq \psi_k+f$ for all $k\geq j$, then $u_j\geq \psi^\prime_j+f$. Letting $j$ to $+\infty$ to get $u\geq \psi+f$. Hence $u\in\mathcal{F}_m(\Omega,f).$ Finally, by \cite[Proposition 2.13]{G2} it follows that 
$$\int_\Omega H_m(u)=\lim_j\int_{\Omega}H_m(u_j)<+\infty.$$
Now if $(u_j)_j\subset\mathcal{F}_m(\Omega,f)$, take  $v_j=\max\{u_j,j\varphi +f\}$ where  $\varphi\in \mathcal{E}_{m}^0(\Omega)$ and $\varphi\neq 0$. Then $v_j\in \mathcal{E}_{m}^0(\Omega,f)$, $v_j\searrow u$ and by \cite[Lemma 2.14]{G2} we have $\sup_j\int_{\Omega}H_m(v_j)\leq \sup_j\int_{\Omega}H_m(u_j)<+\infty$. It follows then from above that
$u \in\mathcal{F}_m(\Omega,f)$
with $\int_{\Omega}H_m(u)<+\infty$.
\end{proof}
\begin{lem}\label{lem2}
    Let $\Omega\subset \tilde{\Omega}$  be bounded $m$-hyperconvex domains in $\mathbb{C}^n$ and let $f\in \mathcal{E}_m(\Omega)$ and $g\in\mathcal{MSH}_m^{-}(\tilde{\Omega})$ be such that $g\leq f$ on $\Omega$. Then for every  $u\in \mathcal{E}_m^0(\Omega,f)$, there exists
a function $\tilde{u}\in \mathcal{E}_m^0(\tilde{\Omega},g)$ such that $\tilde{u}\leq u$, $H_m(\tilde{u})\leq H_m(u)$ on $\Omega$ and $H_m(\tilde{u})=0$ on $\tilde{\Omega}\setminus\Omega.$
\end{lem}
\begin{proof}
    By replacing $g$ by $g-\varepsilon$, $\varepsilon >0$, it follows from \cite[Theorem 3.8]{HP} that we can assume $g+\varepsilon\leq f$ on $\Omega$. Set $$\tilde{u}:=\sup\{\varphi\in\mathcal{SH}_m(\tilde{\Omega}):\; \varphi\leq g\;\mbox{on}\; \tilde{\Omega}\;\mbox{and}\;\varphi\leq u\;\mbox{on}\; \Omega\}.
   $$ 
    We split the proof into two steps.\\
    {\bf Step 1.} We claim that $\tilde{u}\in\mathcal{E}_m^0(\tilde{\Omega},g)$ and $H_m(\tilde{u})=0$ on $\tilde{\Omega}\setminus \Omega$. Indeed, since $u\in \mathcal{E}_m^0(\Omega,f)$, then $\varphi+f\leq u\leq f$ with $\varphi\in\mathcal{E}_m^0(\Omega)$. Put $D=\{\varphi<-\varepsilon\}$. Then $\bar{D}\subset\subset \Omega$ and $u\geq \varphi +f\geq g$ on $\Omega\setminus \bar{D}$. By \cite[Theorem 1]{MV} there exists $\tilde{\varphi}\in \mathcal{F}_m(\tilde{\Omega})$ such that $\tilde{\varphi}\leq \varphi$ on $\Omega$. Let $\psi\in\mathcal{E}_m^0(\tilde{\Omega})$ and $A>>1$ such that $A\psi\leq \varphi$ on $\bar{D}$. Take $\tilde{\psi}=\max\{\tilde{\varphi},A\psi\}$. Then $\tilde{\psi}\in \mathcal{E}_m^0(\tilde{\Omega})$, $\tilde{\psi}+g\leq \varphi +g\leq u$ on $\bar{D}$ and $\tilde{\psi}+g\leq g\leq u$ on $\Omega\setminus \bar{D}$. It follows from the definition of $\tilde{u}$ that $\tilde{\psi}+g\leq \tilde{u}$ on $\tilde{\Omega}$. Hence $\tilde{u}\in\mathcal{E}_m^0(\tilde{\Omega},g)$. \\ On the other hand let  $K\subset \subset\tilde{\Omega}\setminus \bar{D}$ and $w\in\mathcal{SH}_m^-(\tilde{\Omega}\setminus \bar{D})$ such that $w\leq \tilde{u}$ outside $K$. Set
       $$v=\left\{
     \begin{array}{ll}
       \max\{w,\tilde{u}\} & \mbox{on} \ \  \tilde{\Omega}\setminus \bar{D}, \\
        \tilde{u} & \mbox{on}\ \  \bar{D}.
     \end{array}
   \right. $$
      Then $v\in\mathcal{SH}_m^-(\tilde{\Omega})$ and $v\leq \tilde{u}\leq g$ in $\tilde{\Omega}\setminus {K}$. Since $g$ is maximal then $v\leq g$ in $\tilde{\Omega}$. However, $v= \tilde{u}\leq u$ on $\bar{D}$ and $v\leq g\leq u$ on $\Omega\setminus \bar{D}$, then $v\leq u$ on $\Omega$. It follows by definition of $\tilde{u}$ that $v\leq\tilde{u}$ on $\tilde{\Omega}$. Therefore, $ H_m(\tilde{u})=0$ on $\tilde{\Omega}\setminus\Omega.$  \\
      {\bf Step 2.} Now we prove that $H_m(\tilde{u})\leq H_m(u)$ on $\Omega$. We claim that $H_m(\tilde{u})=0$ on $\{\tilde{u}<u\}\cap \Omega$. Indeed, let $(u_j)\subset \mathcal{E}_m^0(\Omega)\cap \mathcal{C}(\bar{\Omega})$ such that $u_j\downarrow u$. Put 
      $$\tilde{u}_j:=\sup\{\varphi\in\mathcal{SH}_m(\tilde{\Omega}):\; \varphi\leq g\;\mbox{on}\; \tilde{\Omega}\;\mbox{and}\;\varphi\leq u_j\;\mbox{on}\; \Omega\}.
   $$ 
   Then $\tilde{u}_j\downarrow \tilde{u}$ and since $\tilde{u}\in\mathcal{E}_m^0(\tilde{\Omega},g)$ we get $\tilde{u}_j\in\mathcal{E}_m^0(\tilde{\Omega},g)$. As $u_j$ is continuous, the set $\{\tilde{u}_j<u_j\}\cap \Omega$ is open. Let $z\in \{\tilde{u}_j<u_j\}\cap \Omega$, then there exists $r>0$ small enough such that $B(z,r)\subset\subset \{\tilde{u}_j<u_j\}\cap \Omega$ and $\displaystyle\sup_{\zeta\in B(z,r)}\tilde{u}_j(\zeta)<\inf_{\zeta\in B(z,r)}{u}_j(\zeta)$. Let $K\subset B(z,r)$ be a compact and $w\in\mathcal{SH}_m^-(B(z,r))$ such that $w\leq \tilde{u}_j$ on $B(z,r)\setminus K$. Set 
   $$v=\left\{
     \begin{array}{ll}
       \tilde{u}_j & \mbox{on} \ \  \tilde{\Omega}\setminus {B}(z,r), \\
        \max\{w,\tilde{u}_j\} & \mbox{on}\ \  B(z,r).
     \end{array}
   \right.$$
   Since $\sup_{\zeta\in \partial B(z,r)}v(\zeta)=\sup_{\zeta\in \partial B(z,r)}\tilde{u}_j(\zeta)$, then $v\leq u_j$ on $B(z,r)$ and by the maximality of $g$ it follows from definition of $\tilde{u}_j$ that $v\leq \tilde{u}_j$ on $\tilde{\Omega}$. Therefore, $w\leq \tilde{u}_j$ on $B(z,r)$. Hence $H_m(\tilde{u}_j)=0$ on $B(z,r)$ and since $z$ is taken arbitrary we get $H_m(\tilde{u}_j)=0$ on $\{\tilde{u}_j<u_j\}\cap \Omega.$  
  Moreover, for all $j\geq k$ we have $\{\tilde{u}_k<u\}\cap \Omega\subset \{\tilde{u}_j<u_j\}\cap \Omega$ since $u\leq u_j$. Hence  $H_m(\tilde{u}_j)=0$ on $\{\tilde{u}_k<u\}\cap \Omega$ for all $j\geq k$. Now since
  $\{\tilde{u}_k<u\}\cap \Omega=\bigcup_{a\in\mathbb{Q}^-}\{\tilde{u}_k<a<u\}\cap \Omega,$ then for all $j\geq k$ we have $H_m(\tilde{u}_j)=0$ on $\{\tilde{u}_k<a<u\}\cap \Omega.$ Furthermore, $\max\{u-a,0\}H_m(\tilde{u}_j)=0$ on $\{\tilde{u}_k<a\}\cap \Omega$,
and by \cite[Theorem 3.8]{HP} we get 
$\max\{u-a,0\}H_m(\tilde{u})=0$ on $\{\tilde{u}_k<a\}\cap \Omega$. Therefore, it follows from \cite[Lemma 4.2]{KH}, that $H_m(\tilde{u})=0$ on $\{\tilde{u}_k<a<u\}\cap \Omega$  and hence on  $\{\tilde{u}_k<u\}\cap \Omega$. Since  $\{\tilde{u}<u\}\cap \Omega=\bigcup_{k\in\mathbb{N}}\{\tilde{u}_k<u\}\cap \Omega,$ then  $H_m(\tilde{u})=0$ on $\{\tilde{u}<u\}\cap \Omega.$ \\ Finally, it remains to prove that $H_m(\tilde{u})\leq H_m(u)$ on $\{\tilde{u}=u\}\cap \Omega.$ Let $K\subset\{\tilde{u}=u\}\cap \Omega $ be a compact. For all $j\geq 1$ we have $K\subset\subset\{\tilde{u}+\frac{1}{j}>u\}\cap \Omega$. It follows from \cite[Theorem 3.6]{HP} that
$$\begin{array}{ll} \displaystyle\int_KH_m(\tilde{u})&=\displaystyle\lim_{j\to+\infty}\int_KH_m\Big(\max\{\tilde{u}+\frac{1}{j},u\}\Big)\\
&\displaystyle\leq \int_KH_m(\max\{\tilde{u},u\})\\
&\displaystyle =\int_KH_m(u).
\end{array}
$$  Therefore, $H_m(\tilde{u})\leq H_m(u)$ on $\Omega$.
\end{proof}
Now we give the proof of our main theorem.
\begin{thm}\label{extension}
  Let $\Omega$ and $\tilde{\Omega}$ be a $m$-hyperconvex domains such that $\Omega\subset\tilde{\Omega}$. Given $f\in\mathcal{E}_m(\Omega)$ and $g\in\mathcal{MSH}_m(\tilde{\Omega})\cap\mathcal{E}_m (\tilde{\Omega})$ that satisfy $g \leq f$ on $\Omega$. If $u\in \mathcal{F}_m(\Omega,f)$ and $\int_{\Omega}H_m(u) < +\infty$,
then there exists a function $\tilde{u}\in\mathcal{F}_m(\tilde{\Omega},g)$ such that $\tilde{u}\leq u$  on $\Omega$, and $H_m(\tilde{u}) =
\mathds{1}_{\Omega}H_m(u).$
\end{thm}
\begin{proof}
We split the proof into two steps.\\
{\bf Step 1.} We claim that there exists $w\in \mathcal{F}_m(\tilde{\Omega},g)$ such that $w\leq u$ on $\Omega$ and $H_m(w)\leq \mathds{1}_\Omega H_m(u)$. Let  $(u_j)\subset\mathcal{E}_m^0(\Omega,f)$ be a sequence decreasing to $u$. By Lemma \ref{lem2}, for all $j$ there exists
a function $\tilde{u}_j\in \mathcal{E}_m^0(\tilde{\Omega},g)$ such that $\tilde{u}_j\leq u_j$, $H_m(\tilde{u}_j)\leq H_m(u_j)$ on $\Omega$ and $H_m(\tilde{u}_j)=0$ on $\tilde{\Omega}\setminus\Omega.$ Moreover, it follows from the construction of each $\tilde{u}_j$ that the sequence $(\tilde{u}_j)_j$ is decreasing. Set $w=\lim_{j\to+\infty}\tilde{u}_j$. By above and \cite[Lemma 2.14]{G2} we have  $w\leq u$ on $\Omega$ and $$\sup_j\int_{\tilde{\Omega}}H_m(\tilde{u}_j)\leq \sup_j\int_{\Omega}H_m(u_j)\leq\int_{\Omega}H_m(u)<+\infty.$$ Therefore, it follows from Lemma \ref{lem1} that $w\in \mathcal{F}_m(\tilde{\Omega},g)$ and $\int_{\tilde{\Omega}}H_m(w)<+\infty.$ Now, since $\mathds{1}_\Omega H_m(u)\leq \lim_j\mathds{1}_\Omega H_m(u_j)$ and $\lim_j\int_{\Omega}H_m(u_j)=\int_{\Omega}H_m(u)$ by \cite[Lemma 2.13]{G2}, then 
$\mathds{1}_\Omega H_m(u)=\lim_j\mathds{1}_\Omega H_m(u_j).$ Therefore we get 
$ H_m(w)=\lim_jH_m(\tilde{u}_j)\leq \lim_j\mathds{1}_\Omega H_m(u_j)=\mathds{1}_\Omega H_m(u)$ on $\tilde{\Omega}.$\\
{\bf Step 2.} It follows from  \cite[Proposition 5.2]{HP} that $\mathds{1}_{\{u=-\infty\}}H_m(u)\leq \mathds{1}_{\{w=-\infty\}\cap \Omega}H_m(w)$. But since $H_m(w)\leq \mathds{1}_\Omega H_m(u)$, then $\mathds{1}_{\{w=-\infty\}}H_m(w)\leq \mathds{1}_{\{w=-\infty\}\cap \Omega}H_m(u)\leq \mathds{1}_{\{u=-\infty\}\cap\Omega}H_m(u)$. Therefore $$\mathds{1}_{\{w=-\infty\}}H_m(w)=\mathds{1}_{\{u=-\infty\}\cap\Omega}H_m(u).$$ Let $\mu=\mathds{1}_{\Omega\cap\{u>-\infty\}}H_m(u)$. Since  $\mu(\tilde{\Omega})\leq \int_\Omega H_m(u)<+\infty$, then by \cite[Theorem 1.2]{Ch1}, there exists $\psi\in\mathcal{F}_m(\tilde{\Omega})$ such that $\mu=H_m(\psi).$ On the other hand, $\mathds{1}_{\{w=-\infty\}}H_m(w)\leq H_m(w)$, then by \cite[Main Theorem]{G2}, there exists $h\in\mathcal{N}_m(\tilde{\Omega},g)$ such that $H_m(h)=\mathds{1}_{\{w=-\infty\}}H_m(w)$ and $w\leq h$ on $\tilde{\Omega}$. Hence $h\in \mathcal{F}_m(\tilde{\Omega},g)$. Put
$$\tilde{u}:=\sup\{\varphi\in\mathcal{E}_m(\tilde{\Omega}):\; \varphi\leq h\; \hbox{and}\; H_m(\varphi)\geq H_m(\psi)\}$$
We have $\psi + h\leq \tilde{u}\leq h$ and since, for all $\xi\in\mathcal{E}_m(\tilde{\Omega})\cap\mathcal{C}(\overline{\tilde{\Omega}})$  we have
$$\begin{array}{ll} \displaystyle\int_{\tilde{\Omega}}(-\xi)(H_m(\psi)+H_m(h))&=\displaystyle\int_{\tilde{\Omega}}(-\xi)\big(\mathds{1}_{\Omega\cap\{u>-\infty\}}H_m(u)+\mathds{1}_{\{w=-\infty\}}H_m(w)\big)\\
&=\displaystyle\int_{\tilde{\Omega}}(-\xi)\big(\mathds{1}_{\Omega\cap\{u>-\infty\}}H_m(u)+\mathds{1}_{\{u=-\infty\}\cap \Omega}H_m(u)\big)\\
&\displaystyle =\int_{{\Omega}}(-\xi)H_m(u)\\
&\displaystyle \leq\sup_{\overline{\tilde{\Omega}}}(-\xi)\int_{{\Omega}}H_m(u)\\
&\displaystyle <+\infty,
\end{array}
$$
then by \cite[Proposition 5.3]{G2} we have $\tilde{u}\in\mathcal{N}_m(\tilde{\Omega},g)$  and $$H_m(\tilde{u})=H_m(\psi)+H_m(h)=\mathds{1}_\Omega H_m(u).$$ 
Since $\psi + h\leq \tilde{u}$ then $\tilde{u}\in\mathcal{F}_m(\tilde{\Omega},g)$. It remains to show that $\tilde{u}\leq u$ on $\Omega.$ Similarly, there exists $\psi_0\in\mathcal{F}_m(\tilde{\Omega})$ such that $H_m(\psi_0)=\mathds{1}_{\{w>-\infty\}}H_m(w)$ and the function
$$w_0:=\sup\{\varphi\in\mathcal{E}_m(\tilde{\Omega}):\; \varphi\leq h\; \hbox{and}\; H_m(\varphi)\geq H_m(\psi_0)\}$$
satisfies $H_m(w_0)=H_m(\psi_0)+H_m(h)=H_m(w).$ Since $w\leq h$, then $w\leq w_0$. Moreover $\int_{\tilde{\Omega}}H_m(w)<+\infty$, it follows from Theorem \ref{31} that $w=w_0$. Finally since $H_m(\tilde{u})\geq H_m(w)\geq H_m(\psi_0)$ and $\tilde{u}\leq h$, then $\tilde{u}\leq w_0=w$, which yields that $\tilde{u}\leq u$ on $\Omega$. The proof is then completed.
\end{proof}

\subsection{A Counterexample}
    We give  an example of an $m$-subharmonic function in the class $\mathcal{N}_m(\Omega)\setminus \mathcal{F}_m(\Omega)$ that can not be extended. We follow the same construction given in \cite{W} and \cite{H} for plurisubharmonic functions. Let $\Omega$ be a bounded $m$-hyperconvex domain and $g_\Omega(z,a)$ be the $m$-Green function for $\Omega$ with a pole at $a\in\Omega$ defined for $1\leq m<n$ by 
    $$
    g^m_\Omega(z,a)=\sup\{u\in\mathcal{SH}_m(\Omega,[-\infty,0));\; u(z)+||z-a||^{2-\frac{2n}{m}}\leq O(1)\;\; \mbox{as}\; z\to a\}.
    $$
    We have $\lim_{z\to \zeta\in\partial\Omega}g_\Omega^m(z,a)=0$, $H_m(g_\Omega^m(.,a))=0$ in $\Omega\setminus\{a\}$ and $H_m(g^m_\Omega(.,a))=\dfrac{(\frac{n}{m}-1)^m}{m!(n-m)!}(2\pi)^n\delta_a$. (See \cite[Theorem A]{EZ} and \cite{WW}). Note that $g^m_\Omega(.,a)\in\mathcal{F}_m(\Omega)$. 
\begin{lem}\label{lemex}
    Let $\Omega$ be a bounded, $m$-hyperconvex domain in $\mathbb{C}^n$ and $w_0\in\partial\Omega$. Then there exists a $m$-polar
set $E\subset\Omega$ such that $\limsup_{w\to w_0}g^m_\Omega(z,w)=0$ for every $z \in\Omega\setminus E$.
\end{lem}
\begin{proof} We proceed as in \cite{CCW}.
  First, let 
$\sigma=(a_j)$ be a sequence in $\Omega$ tending to $w_0\in\partial\Omega$. Define  $h_k(z)=\sup_{j\geq k}g^m_\Omega(z,a_j)$. Since $g^m_\Omega(.,a_k)\leq h_k$, then $h_k^*\in\mathcal{F}_m(\Omega)$ with $\lim_{z\to\zeta\in\partial\Omega}h^*_k(z)=0.$
 Let $\varphi\in\mathcal{E}_m^0(\Omega)\cap\mathcal{C}(\bar{\Omega})$. By integration by parts in $\mathcal{F}_m(\Omega)$ (See \cite[Theorem 3.16]{Ch2}), we get   
   $$
   \begin{array}{ll} 
\displaystyle\int_\Omega-\varphi (dd^ch_k^*)^m\wedge\beta^{n-m}&=\displaystyle\int_\Omega -h_k^*dd^c\varphi\wedge (dd^ch_k^*)^{m-1}\wedge \beta^{n-m}\\ &\leq \displaystyle\int_\Omega -g^m_\Omega(.,a_k)dd^c\varphi\wedge (dd^ch_k^*)^{m-1}\wedge \beta^{n-m}\\ & = \displaystyle\int_\Omega -\varphi dd^cg^m_\Omega(.,a_k)\wedge (dd^ch_k^*)^{m-1}\wedge \beta^{n-m}\\ &\leq \dots\\
&\leq \displaystyle\int_\Omega-\varphi (dd^cg^m_\Omega(.,a_k))^m\wedge\beta^{n-m}\\
&=-\displaystyle\frac{(\frac{n}{m}-1)^m}{m!(n-m)!}(2\pi)^n\varphi(a_k).
   \end{array}
   $$ Therefore, $\displaystyle\lim_{k\to+\infty}\int_\Omega-\varphi (dd^ch_k^*)^m\wedge\beta^{n-m}=0.$ But since $(h_k^*)_k$ is decreasing, then by integration by parts the sequence $\Big(\int_\Omega-\varphi (dd^ch_k^*)^m\wedge\beta^{n-m}\Big)_{k\in\mathbb{N}}$ is increasing. Hence $\int_\Omega-\varphi (dd^ch_k^*)^m\wedge\beta^{n-m}=0$ for all $k$. It follows from \cite[Lemma 3.10]{Ch2} that $H_m(h_k^*)=0$ and hence $h_k^*$ is maximal. Moreover, since $\lim_{z\to\zeta\in\partial\Omega}h^*_k(z)=0$, we see that $h_k^*=0$ on $\Omega$ for all $k$. Thus, for each $k$, there is a $m$-polar set $E_k$ such that $u_k\equiv 0$ on $\Omega\setminus E_k$. Let $E_\sigma=E_1$. Fix $k$ and $z\in\Omega\setminus E_\sigma$. Since $g_\Omega^m(z,a_j)<0$ for each $0\leq j<k$, then $h_k(z)=0$. We conclude that $\limsup_{j\to+\infty}g^m_\Omega(z,a_j)=0$ for every $z \in\Omega\setminus E_\sigma$. Now let $\Sigma$ denotes the set of all sequences in $\Omega$ tending to $w_0$. Take $E=\bigcap_{\sigma\in\Sigma} E_\sigma$, then the set $E$ has the required properties.
\end{proof}
Let $u\in\mathcal{SH}_m(\Omega)$ and $a\in\Omega$. The Lelong number of the $m$-sh function $u$ at
point $a$ is defined by :
$$\nu_u(a)=\lim_{r\to 0}\frac{(\frac{n}{m}-1)^{m-1}}{r^{\frac{2n(m-1)}{m}}}\int_{||z-a||\leq r}dd^cu\wedge\beta^{n-1}.
$$ The number $\nu_u(a)$ exists and is nonnegative (See \cite{WW} and \cite{BG}).
  \begin{thm}
      Let $\Omega$ be a $m$-hyperconvex domain, then there is a function $u\in\mathcal{N}_m(\Omega)\setminus \mathcal{F}_m(\Omega)$ such that $u$ has no subextension.
  \end{thm}  
  \begin{proof}
    Let $w_0\in\partial\Omega$ and  $\rho\in\mathcal{E}_m^0(\Omega)\cap\mathcal{C}(\bar{\Omega})$ be an exhaustive defining function of $\Omega$. By Lemma \ref{lemex}, there exists a $m$-polar set $E$ such that $\limsup_{w\to w_0}g^m_\Omega(z,w)=0$ for every $z \in\Omega\setminus E$. Take $z_0\in\Omega\setminus E$. Then we can find a sequence $(a_j)_j\subset\Omega$ such that $a_j$ tends to $w_0$ as $j\to +\infty,$ $\rho(a_j)>-j^{-2m-1}$ and $g_\Omega^m(z_0,a_j)>-j^{-3}$. Let $u_N=\sum_{j=1}^{N} jg_\Omega^m(.,a_j)$. Then $u_N\in \mathcal{F}_m(\Omega)$. Put $$u(z):=\lim_{N\to +\infty} u_N(z)=\sum_{j=1}^{+\infty} jg_\Omega^m(z,a_j).$$
    Since  $u_N\searrow u$ and $ u(z_0)=\sum_{j=1}^{+\infty} jg_\Omega^m(z_0,a_j)>\sum_{j=1}^{+\infty}-j^{-2}>-\infty,$ then $u\in{SH}^{-}_m(\Omega)$ with $u\not\equiv -\infty.$ Suppose that $u$ has a subextension function $v$ to a domain $\tilde{\Omega}\supset\Omega$ containing  $w_0$. Since $v\leq u$ and $\nu_{u}(a_j)\geq jC_{n,m}$, then $\lim_{j\to+\infty}\nu_v(a_j)=+\infty.$ Moreover by \cite[Corollory 1]{BG}, the Lelong number is upper semi-continuous, then $\nu_v(w_0)=+\infty$. This
is a contradiction since the Lelong number of an $m$-subharmonic function is finite.\\ Now, we must prove that $u\in\mathcal{N}_m(\Omega)\setminus \mathcal{F}_m(\Omega)$. Let $\rho\in\mathcal{E}_m^0(\Omega)\cap\mathcal{C}(\bar{\Omega})$ be an exhaustive function of $\Omega$ and $W\subset\subset\Omega$. Put $$v_N=\sup\{\varphi\in\mathcal{SH}_m(\Omega)\;\; \varphi\leq \max\{u_N,-N\}\;\mbox{on}\; W\}.$$ Then $v_N\in \mathcal{E}_m^0(\Omega)$, $v_N\geq\max\{u_N,-N\}$,  $v_N\searrow u$ on $W$ and $H_m(v_N)$ is supported in $\overline{W}$. Let $A=(\inf_{\overline{W}}-\rho)^{-1}$. Therefore, by \cite[Lemma 2.14]{G2} and \cite[Corollary 1]{HZ} we get
$$\begin{array}{ll}\displaystyle\sup_N\int_\Omega H_m(v_N)&\leq 
\displaystyle \sup_NA\int_\Omega -\rho H_m(v_N)\\ &\displaystyle\leq \sup_NA\int_\Omega-\rho H_m(\max\{u_N,-N\})\\ &\displaystyle\leq \sup_NA\int_\Omega -\rho H_m(u_N)\\
&\displaystyle\leq \sup_NA\Big[\sum_{j=1}^N\Big(\int_\Omega -\rho jH_m(g_\Omega^m(z,a_j))\Big)^{\frac{1}{m}}\Big]^{m}\\
&=\displaystyle\sup_NA \Big[\sum_{j=1}^N\Big(\int_\Omega -\rho j\frac{(\frac{n}{m}-1)^m}{m!(n-m)!}(2\pi)^n\delta_{a_j}\Big)^{\frac{1}{m}}\Big]^{m}\\
&= \displaystyle\sup_NA \Big[\sum_{j=1}^N\Big( - j\frac{(\frac{n}{m}-1)^m}{m!(n-m)!}(2\pi)^n\rho(a_j)\Big)^{\frac{1}{m}}\Big]^{m}\\
&\leq AC_{n,m}(2\pi)^n\lim_N\displaystyle \Big[\sum_{j=1}^N \frac{1}{j^2}\Big]^{m}\\
&\displaystyle= AC_{n,m}\frac{2^n}{6^m}\pi^{n+2m}\\
&<+\infty.
\end{array}
$$
Thus $u\in\mathcal{E}_m(\Omega)$. Moreover, $u$ is in the form $\sum_{j}^{+\infty}v_j$ with $v_j\in\mathcal{F}_m(\Omega)$. We can then apply \cite[Theorem 4.6]{V} to get $u\in\mathcal{N}_m(\Omega)$. But since $\int_\Omega H_m(u)=+\infty$, then $u\not\in\mathcal{F}_m(\Omega).$
\end{proof}
\section{Approximation in the class $\mathcal{F}_m(\Omega,f)$}\label{sec4}
In this section we prove the approximation Theorem.
\begin{thm}\label{app}
Let $\Omega\subset\Omega_{j+1}\subset\Omega_j$
be $m$-hyperconvex domains such that $\lim_{j\to +\infty}\mathrm{cap}_m(\Omega_j,K)=\mathrm{cap}_m(\Omega,K)$ for all compact subset
$K\subset\Omega$, and $g\in \mathcal{MSH}_m^{-}(\Omega_1)$. Then, to every function
$u\in \mathcal{F}_m(\Omega,g_{|\Omega})$, such that
$\int_{\Omega}H_m(u) < +\infty$,
there exists an increasing sequence of functions $u_j\in\mathcal{F}(\Omega_j,g_{|\Omega_j})$ such that
$\lim_{j\to +\infty} u_j=u$ almost everywhere on $\Omega$.
\end{thm}
\begin{proof}
    Let $\{\Omega_j\}$ be a decreasing sequence of $m$-hyperconvex domains containing $\Omega$,
 $g$ be a negative function in $\mathcal{MSH}_m(\Omega_1)$,  and  let  $u\in \mathcal{F}_m(\Omega,f)$  such that
 $\int_{\Omega}H_m(u) < +\infty$ where $f=g_{|\Omega}$. Let $f_j=g_{|\Omega_j}$. Since $u\in\mathcal{F}_m(\Omega,f)$,
then there exists $\psi\in\mathcal{F}_m(\Omega)$ such that $f\geq u\geq \psi+ f$. \\
It follows from \cite[Theorem 4.1. $(v)\to (ii)$]{PD} that there exists an increasing sequence
$(\psi_j)_j$ such that $\psi_j\in\mathcal{F}_m(\Omega_j)$ and $\lim_j\psi_j=\psi$ quasi everywhere on $\Omega$. Moreover, it follows from from steps 1 in the proof of Theorem \ref{extension}, that 
the functions $u_j$ defined by
$$
u_j:=\sup\{\varphi\in \mathcal{E}_m(\Omega_j)/\; \varphi\leq f_j\;\mbox{on}\;\Omega_j\;\,\mbox{and}\;\varphi\leq u\;\mbox{on}\;\Omega\}$$
belongs to $u_j\in \mathcal{F}_m(\Omega_j,f_j)$, with $u_j\leq u$ and $H_m(u_j)\leq \mathds{1}_{\Omega}H_m(u)$. Since $\psi_j+f_j\leq f_j$
and $\psi_j+f_j\leq \psi+f\leq u$ on  $\Omega$, then $
  \psi_j+f_j\leq u_j\leq f_j $ for all $j$. Put $h=(\lim_ju_j)^*$. From above $h\in\mathcal{F}_m(\Omega,f)$
  and since $(u_j)_j$ is increasing, then by \cite[Theorem 3.8]{HP} we get $H_m(h)\leq H_m(u)$.
    Moreover, since $h\leq u$ and $\int_\Omega H_m(h)<+\infty$,
    then by \cite[Lemma 2.14]{G2}, we have $\int_\Omega H_m(u)\leq\int_\Omega H_m(h)$ hence
    $ H_m(h)= H_m(u)$. Finally it follows from Theorem \ref{31}  that $h=u$ and the Theorem is proved.
\end{proof}
In the following example, which is proven in \cite[Theorem 4.3]{PD}, we give a sufficient condition for $\Omega$ to get Theorem \ref{app}.
\begin{exa}
  Let $\Omega\subset\subset\mathbb{C}^n$ be a strongly $m$-hyperconvex domain and $\{\Omega_j\}$ be a
decreasing sequence of $m$-hyperconvex domains such that $\Omega=\big(\bigcap_j\Omega_j\big)^{\circ}$, then $\lim_{j\to +\infty}\mathrm{cap}_m(\Omega_j,K)=\mathrm{cap}_m(\Omega,K)$ for all compact subset $K\subset\Omega$. 
 \end{exa}
In the following theorem, we give another example of domain $\Omega$ in which the condition of Theorem \ref{app} is satisfied. A bounded domain $\Omega$ in $\mathbb{C}^n$ is called $B_m$-regular if for every continuous function $\varphi$ on $\partial\Omega$ we can find a function $u\in\mathcal{SH}_m(\Omega)\cap\mathcal{C}(\bar{\Omega})$ such that $u_{|\partial\Omega}=\varphi$.
\begin{thm}\label{example}
    Let $\Omega$ be the intersection of a finite number of bounded $B_m$-regular domains with $\mathcal{C}^1$-smooth boundary and $\{\Omega_j\}$ be a
decreasing sequence of $m$-hyperconvex domains such that $\bar{\Omega}=\bigcap_j\Omega_j$, then $\lim_{j\to +\infty}\mathrm{cap}_m(\Omega_j,K)=\mathrm{cap}_m(\Omega,K)$ for all compact subset $K\subset\Omega$.
\end{thm}
\begin{proof}
  Let $B$ be a closed ball in $\Omega$. Then by \cite[Proposition 3.2]{Blo} (see also the Remark after \cite[Definition 2.4 and Theorem 2.5]{ACH}), the  extremal function $u:=u_{m,B,\Omega}\in\mathcal{E}_0(\Omega)\cap\mathcal{C}(\bar{\Omega})$. By \cite[Corollary 1.9]{DHTO}, there exists $(u_k)$ a sequence of continuous $m$-subharmonic functions on neighborhoods of $\bar{\Omega}$ such that $u_k\downarrow u$ on $\bar{\Omega}$. By Dini Theorem, this convergence is uniform. Take $\varepsilon>0$, then there exists $k_0$ such that for all $k\geq k_0$ we have $\sup_{\bar{\Omega}}|u-u_k|<\varepsilon.$ Since $\bar{\Omega}=\bigcap_j\Omega_j$, we can  take a large $j_k$ such that $u_k\in\mathcal{SH}_m(\Omega_{j_k})\cap\mathcal{C}(\bar{\Omega}_{j_k}).$ Put
  $$v_{j_k}(z)=\sup\{\varphi(z)\,:\, \varphi\in\mathcal{SH}^{-}_m(\Omega_{j_k})\; \mbox{and}\;
  \varphi\leq u\;\mbox{on}\;\Omega\}.
  $$ The sequence $(v_{j_k})$ is increasing and by the proof of Lemma \ref{lem2} we have $v_{j_k}\in\mathcal{E}_m^0(\Omega)$. Since $u_k-\varepsilon<u$ on $\Omega$, then $u_k-\varepsilon\leq v_{j_k}$. It follows that $u-\varepsilon\leq \lim v_{j_k}\leq u$ and hence $\lim v_{j_k}=u$. The desired conclusion  follows from \cite[Theorem 4.1 $(iii)\to (iv)$]{PD}.
\end{proof}

\begin{cor}
    Let $\Omega$ be a strictely $m$-pseudoconvex domain. Then there exists a sequence of $m$-hyperconvex domains $\Omega_j\supset\Omega_{j+1}\supset\Omega$ such that for $g\in \mathcal{MSH}_m^{-}(\Omega_1)$, every function $u\in\mathcal{F}_m(\Omega,g_{|\Omega})$ satisfying $\int_{\Omega}H_m(u) < +\infty$
 can be approximated almost everywhere on $\Omega$ by an increasing sequence of functions  $u_j\in\mathcal{F}_m(\Omega_j,g_{|\Omega_j})$.
\end{cor}
\begin{proof}
  Let $\rho$ be a smooth strictly $m$-sh function  on some open neighborhood $\Omega^{\prime}$ of $\bar \Omega$ such that $\Omega= \{z\in\Omega^{\prime}:\; \rho(z) < 0\}$. Take $\rho_j=\rho-\dfrac{1}{j}$ and $\Omega_j= \{z\in\Omega^{\prime}:\; \rho_j(z) < 0\}$. Then the Corollary follows from Theorem \ref{app} and Theorem \ref{example}.
\end{proof}

\noindent{\bf Acknowledgement}. We would like to thank  the referees
for their careful reading as well as for their comments and suggestions
that have contributed to improving the readability and quality of the
paper.
\author{ }
\date{September 2023}

\end{document}